\documentclass{amsart}
\usepackage{amssymb}

\title[On Hall-Littlewood polynomials]{Hall-Littlewood Polynomials, Alcove Walks, and Fillings of Young Diagrams}

\author{Cristian Lenart}

\address{Department of Mathematics and Statistics, State University of New York at Albany, Albany, NY 12222, USA; Phone +1-518-442-4635, Fax +1-518-442-4731}
\email{lenart@albany.edu}

\keywords{Hall-Littlewood polynomials, Macdonald polynomials, alcove walks, Schwer's formula, the Haglund-Haiman-Loehr formula, fillings of Young diagrams.}

\subjclass[2000]{Primary 05E05. Secondary 33D52.}

\thanks{Cristian Lenart was partially supported by the National Science Foundation grant  DMS-0701044}

\DeclareMathOperator{\des}{des}

\DeclareMathOperator{\inv}{inv}
\DeclareMathOperator{\cinv}{cinv}
\DeclareMathOperator{\rt}{r}
\DeclareMathOperator{\rev}{rev}
\DeclareMathOperator{\ct}{ct}

\newlength{\cellsize}
\newcommand\tableau[1]{
\vcenter{
\let\\=\cr
\baselineskip=-16000pt
\lineskiplimit=16000pt
\lineskip=0pt
\halign{&\tableaucell{##}\cr#1\crcr}}}

\newcommand{\tableaucell}[1]{{%
\def \arg{#1}\def \void{}%
\ifx \void \arg
\vbox to \cellsize{\vfil \hrule width \cellsize height 0pt}%
\else
\unitlength=\cellsize
\begin{picture}(1,1)
\put(0,0){\makebox(1,1){$#1$}}
\put(0,0){\line(1,0){1}}
\put(0,1){\line(1,0){1}}
\put(0,0){\line(0,1){1}}
\put(1,0){\line(0,1){1}}
\end{picture}%
\fi}}

\numberwithin{equation}{section}

\theoremstyle{plain}
\newtheorem{theorem}{Theorem}[section]
\newtheorem{proposition}[theorem]{Proposition}
\newtheorem{lemma}[theorem]{Lemma}

\newtheorem{definition}[theorem]{Definition}

\newtheorem{example}[theorem]{Example}

\theoremstyle{remark}
\newtheorem{remark}[theorem]{Remark}

\def\R{\mathbb{R}}
\def\Z{\mathbb{Z}}

\def\F{{\mathcal F}}
\def\Fh{ {\overline{\mathcal{F} } } }
\def\A{\mathcal{A}}

\def\Waff{W_{\mathrm{aff}}}

\def\h{\mathfrak{h}}
\def\hR{\mathfrak{h}^*_\mathbb{R}}

\newcommand{\stacksum}[2]{\sum_{\begin{array}{c}\vspace{-5.4mm}\;\\ \vspace{-1mm}\scriptstyle{#1}\\ \scriptstyle{#2}\end{array}} }

\begin{document}
\bibliographystyle{plain}

\begin{abstract} 
A recent breakthrough in the theory of (type $A$) Macdonald polynomials is due to Haglund, Haiman and Loehr, who exhibited a combinatorial formula for these polynomials in terms of a pair of statistics on fillings of Young diagrams. The inversion statistic, which is the more intricate one, suffices for specializing a closely related formula to one for the type $A$ Hall-Littlewood $Q$-polynomials (spherical functions on $p$-adic groups). An apparently unrelated development, at the level of arbitrary finite root systems, led to Schwer's formula (rephrased and rederived by Ram) for the Hall-Littlewood $P$-polynomials of arbitrary type. The latter formula is in terms of so-called alcove walks, which originate in the work of Gaussent-Littelmann and of the author with Postnikov on discrete counterparts to the Littelmann path model. In this paper, we relate the above developments, by deriving a Haglund-Haiman-Loehr type formula for the Hall-Littlewood $P$-polynomials of type $A$ from Ram's version of Schwer's formula via a ``compression'' procedure. 
\end{abstract}

\maketitle


\section{Introduction}
\label{intro}

Hall-Littlewood symmetric polynomials are at the center of many recent developments in representation theory and algebraic combinatorics. They were originally defined in type $A$, as a basis for the algebra of symmetric functions depending on a parameter $t$; this basis interpolates between two fundamental bases: the one of Schur functions, at $t=0$, and the one of monomial functions, at $t=1$. The original motivation for defining Hall-Littlewood polynomials comes from some counting problems in group theory that led to the definition of the Hall algebra \cite{litcsf}. Apart from the Hall algebra, the best known applications of Hall-Littlewood functions include: 
\begin{itemize}
\item the character theory of finite linear groups \cite{grecfg}, 
\item spherical functions for the general linear group over the field of $p$-adic numbers \cite{macsfh}[Chapter V], 
\item projective and modular representations of symmetric groups \cite{morasf,schuds}, 
\item Lusztig's $t$-analog of weight multiplicities (Kostka-Foulkes polynomials) and affine Hecke algebras \cite{lusscq}, 
\item unipotent classes and Springer representations \cite{hasstc,lusgps}, 
\item statistical physics related to certain Bethe ansatz configurations and fermionic multiplicity formulas \cite{karbac},
\item representations of quantum affine algebras and affine crystals \cite{lltrth,laslqa}.
\end{itemize} 

Macdonald \cite{macsfg} showed that there is a formula for the spherical functions corresponding to a Chevalley group over a $p$-adic field which generalizes the formula for the Hall-Littlewood polynomials. Thus, the Macdonald spherical functions generalize the Hall-Littlewood polynomials to all root systems, and the two names are used interchangeably in the literature. There are two families of Hall-Littlewood polynomials of arbitrary type, called $P$-polynomials and $Q$-polynomials, which form dual bases for the Weyl group invariants. The $P$-polynomials specialize to the Weyl characters at $t=0$.  The transition matrix between Weyl characters and $P$-polynomials is given by Lusztig's $t$-analog of weight multiplicities (Kostka-Foulkes polynomials of arbitrary type), which are certain affine Kazhdan-Lusztig polynomials \cite{katsfq,lusscq}. On the combinatorial side, we have the Lascoux-Sch\"utzenberger formula for the Kostka-Foulkes polynomials in type $A$ \cite{lassuc}, but no generalization of this formula to other types is known. Other applications of the type $A$ Hall-Littlewood polynomials that extend to arbitrary type are those related to fermionic multiplicity formulas \cite{aakfpk} and affine crystals \cite{laslqa}. We refer to \cite{dlthlf,macsfh,narkfp,stekfp} for surveys on Hall-Littlewood polynomials, both of type $A$ and of arbitrary type. 

Macdonald \cite{macsft,macopa} defined a remarkable family of orthogonal polynomials depending on parameters $q,t$, which bear his name. These polynomials generalize the spherical functions for a $p$-adic group, the Jack polynomials, and the zonal polynomials. At $q=0$, Macdonald's integral form polynomials $J_\lambda(X;q,t)$ specialize to the Hall-Littlewood $Q$-polynomials, and thus they further specialize to the Weyl characters (upon setting $t=0$ as well). There has been considerable interest recently in the combinatorics of Macdonald polynomials. This stems in part from a combinatorial formula for the ones corresponding to type $A$, which is due to Haglund, Haiman, and Loehr \cite{hhlcfm}, and which is in terms of fillings of Young diagrams. This formula uses two statistics on the mentioned fillings, called inv and maj. The Haglund-Haiman-Loehr (HHL) formula already found important applications, such as new proofs of the positivity theorem for Macdonald polynomials, which states that the two-parameter Kostka-Foulkes polynomials have nonnegative integer coefficients. One of the mentioned proofs, due to Grojnowski and Haiman \cite{gahaha}, is based on Hecke algebras, while the other, due to Assaf \cite{asasem}, is purely combinatorial and leads to a positive formula for the two-parameter Kostka-Foulkes polynomials.

An apparently unrelated development, at the level of arbitrary finite root systems, led to Schwer's formula \cite{schghl}, rephrased and rederived by Ram \cite{ramawh}, for the Hall-Littlewood $P$-polynomials of arbitrary type. The latter formula is in terms of so-called alcove walks, which originate in the work of Gaussent-Littelmann \cite{gallsg} and of the author with Postnikov \cite{lapawg,lapcmc} on discrete counterparts to the Littelmann path model in the representation theory of Lie algebras \cite{litlrr,litpro}. 

In this paper, we relate Schwer's formula to the HHL formula. More precisely, we show that, if the partition $\lambda$ has $n-1$ distinct non-zero parts, then we can group into equivalence classes the terms in the type $A_{n-1}$ instance of Ram's version of Schwer's formula for $P_\lambda(X;t)$, such that the sum in each equivalence class is a term in the HHL formula for $q=0$. An equivalence class consists of all the terms corresponding to alcove walks that produce the same filling of the Young diagram $\lambda$ via a simple construction. In consequence, we explain the way in which the Macdonald polynomial inversion statistic (which is the more intricate of the two statistics mentioned above) arises naturally from more general concepts, as the outcome of ``compressing'' Ram's version of Schwer's formula in type $A$. More generally, when $\lambda$ is arbitrary, with A. Lubovsky we used a combinatorial bijection to show that (Ram's version of) Schwer's formula leads to a new HHL-type formula for $P_\lambda(X;t)$; this result is contained in the Appendix.

This article lays the groundwork for several directions of research. In \cite{lencfm} we show that the recent  formula for the Macdonald polynomials due to Ram and Yip \cite{raycfm} (which is also in terms of alcove walks, but does not specialize to Ram's version of Schwer's formula used in this paper upon setting $q=0$) compresses to a formula which is similar to the HHL one, but has fewer terms. Thus, the results in this paper are not a specialization of those in \cite{lencfm}. Furthermore, note that, unlike in the present paper, in \cite{lencfm} we only consider partitions $\lambda$ with $n-1$ distinct non-zero parts, and that certain key facts needed in \cite{lencfm} are proved in the present paper. 
In \cite{lenhhl}, we derive new tableau formulas for the Hall-Littlewood polynomials of type $B$ and $C$ by compressing the corresponding instances of (Ram's version of) Schwer's formula. Type $D$ is slightly more complex, and will be considered in a different publication. We are also  investigating potential applications to positive combinatorial formulas for Lusztig's $t$-analog of weight multiplicity beyond type $A$.

\section{Preliminaries}\label{prelim}

We recall some background information on finite root systems and affine Weyl groups.

\subsection{Root systems}\label{rootsyst}

Let $\mathfrak{g}$ be a complex semisimple Lie algebra, and $\h$ a Cartan subalgebra, whose rank is $r$.
Let $\Phi\subset \h^*$ be the 
corresponding irreducible {\it root system}, $\hR\subset \h^*$ the real span of the roots, and $\Phi^+\subset \Phi$ the set of positive roots. 
Let $\alpha_1,\ldots,\alpha_r\in\Phi^+$ be the corresponding 
{\it simple roots}.
We denote by $\langle\,\cdot\,,\,\cdot\,\rangle$ the nondegenerate scalar product on $\hR$ induced by
the Killing form.  
Given a root $\alpha$, we consider the corresponding {\it coroot\/} $\alpha^\vee := 2\alpha/\langle\alpha,\alpha\rangle$ and reflection $s_\alpha$.  

Let $W$ be the corresponding  {\it Weyl group\/}, whose Coxeter generators are denoted, as usual, by $s_i:=s_{\alpha_i}$. The length function on $W$ is denoted by $\ell(\,\cdot\,)$. The {\em Bruhat graph} on $W$ is the directed graph with edges $u\rightarrow w$ where $w=u s_{\beta}$ for some $\beta\in\Phi^+$, and $\ell(w)>\ell(u)$; we usually label such an edge by $\beta$ and write $u\stackrel{\beta}\longrightarrow w$. The {\em reverse Bruhat graph} is obtained by reversing the directed edges above. The {\em Bruhat order} on $W$ is the transitive closure of the relation corresponding to the Bruhat graph.

The {\it weight lattice\/} $\Lambda$ is given by
\begin{equation}
\Lambda:=\{\lambda\in \hR \::\: \langle\lambda,\alpha^\vee\rangle\in\Z
\textrm{ for any } \alpha\in\Phi\}.
\label{eq:weight-lattice}
\end{equation}
The weight lattice $\Lambda$ is generated by the 
{\it fundamental weights\/}
$\omega_1,\ldots,\omega_r$, which form the dual basis to the 
basis of simple coroots, i.e., $\langle\omega_i,\alpha_j^\vee\rangle=\delta_{ij}$.
The set $\Lambda^+$ of {\it dominant weights\/} is given by
$$
\Lambda^+:=\{\lambda\in\Lambda \::\: \langle\lambda,\alpha^\vee\rangle\geq 0
\textrm{ for any } \alpha\in\Phi^+\}.
$$
The subgroup of $W$ stabilizing a weight $\lambda$ is denoted by $W_\lambda$, and the set of minimum coset representatives in $W/W_\lambda$ by $W^\lambda$. Let $\Z[\Lambda]$ be the group algebra of the weight lattice $\Lambda$, which  has
a $\Z$-basis of formal exponents $\{x^\lambda \::\: \lambda\in\Lambda\}$ with
multiplication $x^\lambda\cdot x^\mu := x^{\lambda+\mu}$.

Given  $\alpha\in\Phi$ and $k\in\Z$, we denote by $s_{\alpha,k}$ the reflection in the affine hyperplane
\begin{equation}
H_{\alpha,k} := \{\lambda\in \hR \::\: \langle\lambda,\alpha^\vee\rangle=k\}.
\label{eqhyp}
\end{equation}
These reflections generate the {\it affine Weyl group\/} $\Waff$ for the {\em dual root system} 
$\Phi^\vee:=\{\alpha^\vee \::\: \alpha\in\Phi\}$. 
The hyperplanes $H_{\alpha,k}$ divide the real vector space $\hR$ into open
regions, called {\it alcoves.} 
The {\it fundamental alcove\/} $A_\circ$ is given by 
$$
A_\circ :=\{\lambda\in \hR \::\: 0<\langle\lambda,\alpha^\vee\rangle<1 \textrm{ for all }
\alpha\in\Phi^+\}.
$$

\subsection{Alcove walks}\label{alcovewalks}

We say that two alcoves $A$ and $B$ are {\it adjacent} 
if they are distinct and have a common wall.  
Given a pair of adjacent alcoves $A\ne B$ (i.e., having a common wall), we write 
$A\stackrel{\beta}\longrightarrow B$ if the common wall 
is of the form $H_{\beta,k}$ and the root $\beta\in\Phi$ points 
in the direction from $A$ to $B$.  

\begin{definition}
An {\em alcove path\/} is a sequence of alcoves
 such that any two consecutive ones are adjacent. 
We say that an alcove path $(A_0,A_1,\ldots,A_m)$ is {\it reduced\/} if $m$ is the minimal 
length of all alcove paths from $A_0$ to $A_m$.
\end{definition}

We need the following generalization of alcove paths.

\begin{definition}\label{defalcwalk} An {\em alcove walk} is a sequence 
$\Omega=(A_0,F_1,A_1, F_2, \ldots , F_m, A_m, F_{\infty})$ 
such that $A_0,\ldots,$ $A_m$ are alcoves; 
$F_i$ is a codimension one common face of the alcoves $A_{i-1}$ and $A_i$,
for $i=1,\ldots,m$; and 
$F_{\infty}$ is a vertex of the last alcove $A_m$. The weight $F_\infty$ is called the {\em weight} of the alcove walk, and is denoted by $\mu(\Omega)$. 
\end{definition}
 
The {\em folding operator} $\phi_i$ is the operator which acts on an alcove walk by leaving its initial segment from $A_0$ to $A_{i-1}$ intact and by reflecting the remaining tail in the affine hyperplane containing the face $F_i$. In other words, we define
$$\phi_i(\Omega):=(A_0, F_1, A_1, \ldots, A_{i-1}, F_i'=F_i,  A_{i}', F_{i+1}', A_{i+1}', \ldots,  A_m', F_{\infty}')\,;$$
here $A_j' := \rho_i(A_j)$ for $j\in\{i,\ldots,m\}$, $F_j':=\rho_i(F_j)$ for $j\in\{i,\ldots,m\}\cup\{\infty\}$, and $\rho_i$ is the affine reflection in the hyperplane containing $F_i$. Note that any two folding operators commute. An index $j$ such that $A_{j-1}=A_j$ is called a {\em folding position} of $\Omega$. Let $\mbox{fp}(\Omega):=\{ j_1<\ldots< j_s\}$ be the set of folding positions of $\Omega$. If this set is empty, $\Omega$ is called {\em unfolded}. Given this data, we define the operator ``unfold'', producing an unfolded alcove walk, by
\[\mbox{unfold}(\Omega)=\phi_{j_1}\ldots \phi_{j_s} (\Omega)\,.\]

\begin{definition} An  alcove walk
$\Omega=(A_0,F_1,A_1, F_2, \ldots , F_m, A_m, F_{\infty})$ is called {\em positively folded} if, for any folding position $j$, the alcove $A_{j-1}=A_j$ lies on the positive side of the affine hyperplane containing the face $F_j$.
\end{definition} 

We now fix a dominant weight $\lambda$ and a reduced alcove path $\Pi:=(A_0,A_1,\ldots,A_m)$ from $A_\circ=A_0$ to its translate $A_\circ + \lambda=A_m$. Assume that we have
\[A_0\stackrel{\beta_1}\longrightarrow A_1\stackrel{\beta_2}\longrightarrow \ldots
\stackrel{\beta_m}\longrightarrow A_{m}\,,\]
where $\Gamma:=(\beta_1,\ldots,\beta_m)$ is a sequence of positive roots. This sequence, which determines the alcove path, is called a {\em $\lambda$-chain} (of roots). Two equivalent definitions of $\lambda$-chains (in terms of reduced words in affine Weyl groups, and an interlacing condition) can be found in \cite{lapawg}[Definition 5.4] and \cite{lapcmc}[Definition 4.1 and Proposition 4.4]; note that the $\lambda$-chains considered in the mentioned papers are obtained by reversing the ones in the present paper. We also let $r_i:=s_{\beta_i}$, and let $\widehat{r}_i$ be the affine reflection in the common wall of $A_{i-1}$ and $A_i$, for $i=1,\ldots,m$; in other words, $\widehat{r}_i:=s_{\beta_i,l_i}$, where $l_i:=|\{j\le i\::\: \beta_j = \beta_i\}|$ is the cardinality of the corresponding set. Given $J=\{j_1<\ldots<j_s\}\subseteq[m]:=\{1,\ldots,m\}$, we define the Weyl group element $\phi(J)$ and the weight $\mu(J)$ by
\begin{equation}\label{defphimu}\phi(J):={r}_{j_1}\ldots {r}_{j_s}\,,\;\;\;\;\;\mu(J):=\widehat{r}_{j_1}\ldots \widehat{r}_{j_s}(\lambda)\,.\end{equation}

 Given $w\in W$, we define the alcove path $w(\Pi):=(w(A_0),w(A_1),\ldots,w(A_m))$. Consider the set of alcove paths
\[{\mathcal P}(\Gamma):=\{w(\Pi)\::\:w\in W^\lambda\}\,.\]
We identify any $w(\Pi)$ with the obvious unfolded alcove walk of weight $\mu(w(\Pi)):=w(\lambda)$. Let us now consider the set of alcove walks
\[{\mathcal F}_+(\Gamma):=\{\,\mbox{positively folded alcove walks $\Omega$}\::\:\mbox{unfold}(\Omega)\in{\mathcal P}(\Gamma)\}\,.\]
We can encode an alcove walk $\Omega$ in ${\mathcal F}_+(\Gamma)$ by the pair $(w,J)$ in $W^\lambda\times 2^{[m]}$, where 
\[\mbox{fp}(\Omega)=J\;\;\;\;\mbox{and}\;\;\;\;\mbox{unfold}(\Omega)=w(\Pi)\,.\]
Clearly, we can recover $\Omega$ from $(w,J)$ with $J=\{j_1<\ldots<j_s\}$ by
\[\Omega=\phi_{j_1}\ldots \phi_{j_s} (w(\Pi))\,.\]
Let ${\mathcal A}(\Gamma)$ be the image of ${\mathcal F}_+(\Gamma)$ under the map $\Omega\mapsto (w,J)$. We call a pair $(w,J)$ in ${\mathcal A}(\Gamma)$ an {\em admissible pair}, and the subset $J\subseteq[m]$ in this pair a $w$-{\em admissible subset}. 

\begin{proposition}\label{admpairs} {\rm (1)} We have
\begin{equation}\label{decch}{\mathcal A}(\Gamma)=\{\,(w,J)\in W^\lambda\times 2^{[m]}\::\: J=\{j_1<\ldots<j_s\}\,,\;\;w>wr_{j_1}>\ldots>wr_{j_1}\ldots r_{j_s}=w\phi(J)\}\,;\end{equation}
here the decreasing chain is in the Bruhat order on the Weyl group, its steps not being covers necessarily.

{\rm (2)} If $\Omega\mapsto (w,J)$, then
\[\mu(\Omega)=w(\mu(J))\,.\] 
\end{proposition}

\begin{proof} The first part rests on the well-known fact that, given a positive root $\alpha$ and a Weyl group element $w$, we have $\ell(ws_\alpha)<\ell(w)$ if and only if $w(\alpha)$ is a negative root \cite[Proposition~5.7]{humrgc}. The second part follows from the simple fact that the action of $w$ on an alcove walk commutes with that of the folding operators. 
\end{proof}

The formula for the Hall-Littlewood $P$-polynomials in \cite{schghl} was rederived in \cite{ramawh} in a slightly different version, based on positively folded alcove walks. Based on Proposition \ref{admpairs}, we now restate the latter formula in terms of admissible pairs. 

\begin{theorem}\cite{ramawh,schghl} \label{hlpthm} Given a dominant weight $\lambda$, we have
\begin{equation}\label{hlpform}P_{\lambda}(X;t)=\sum_{(w,J)\in{\mathcal A}(\Gamma)}t^{\frac{1}{2}(\ell(w)+\ell(w\phi(J))-|J|)}\,(1-t)^{|J|}\,x^{w(\mu(J))}\,.\end{equation}
\end{theorem}

\subsection{The Macdonald polynomial inversion statistic} 

This subsection recalls the setup in \cite{hhlcfm}, closely following the presentation there.

Let $\lambda = (\lambda _{1}\geq \lambda _{2}\geq \ldots \geq \lambda _{l})$ with $\lambda_l>0$ be a
partition of $m = \lambda _{1}+\ldots +\lambda _{l}$. The number of parts $l$ is known as the {\em length} of $\lambda$, and is denoted by $\ell(\lambda)$. Using standard notation, one defines
\[n(\lambda):=\sum_{i}(i-1)\lambda_i\,.\]
We identify $\lambda$ with its Young (or Ferrers) diagram
\begin{equation*}
 \{(i,j)\in  \Z _{+}\times \Z _{+}: j\leq \lambda _{i}\}\,,
\end{equation*}
whose elements are called {\it
cells}.   Diagrams are drawn in ``Japanese style'' (i.e., in the third quadrant), as shown below:
\begin{equation*}
\lambda =(2,2,2,1) = \tableau{{}&{}\\ {}&{}\\{}&{}\\&{}}\;;
\end{equation*}
the rows and columns are increasing in the negative direction of the axes. We denote, as usual, by $\lambda'$ the conjugate partition of $\lambda$ (i.e., the reflection of the diagram of $\lambda$ in the line $y=-x$, which will be drawn in French style). For any cell $u=(i,j)$ of $\lambda$ with $j\ne 1$, denote the cell $v=(i,j-1)$ directly to the right of $u$ by $\rt(u)$. 

Two cells $u,v\in \lambda $ are said to {\it attack} each other if either
\begin{itemize}
\item [(i)] they are in the same column: $u = (i,j)$, $v = (k,j)$; or
\item [(ii)] they are in consecutive columns, with the cell in the left column
strictly below the one in the right column: $u=(i,j)$,
$v=(k,j-1)$, where $i>k$.
\end{itemize}
The figure below shows the two types of pairs of attacking cells.
\begin{equation*}
\text{(i)}\quad \tableau{{\bullet}&{}\\ {}&{}\\{\bullet}&{}\\&{}} \, ,\qquad
\text{(ii)}\quad  \tableau{{}&{\bullet}\\ {}&{}\\{\bullet}&{}\\&{}}\; .
\end{equation*}

A {\it filling} is a function
$\sigma \::\: \lambda \rightarrow [n]:=\{1,\ldots,n\}$ for some $n$, that is, an assignment of values in $[n]$ to the cells of $\lambda$.  As usual, we define the content of a filling $\sigma$ as ${\rm ct}(\sigma):=(c_1,\ldots,c_n)$, where $c_i$ is the number of entries $i$ in the filling, i.e., $c_i:=|\sigma^{-1}(i)|$. The monomial $x^{{\rm ct}(\sigma)}$ of degree $m$ in the variables $x_{1},\ldots,x_n$ is then given by
\begin{equation*}\label{e:xsigma}
x^{{\rm ct}(\sigma)} := x_1^{c_1}\ldots x_n^{c_n}\,.
\end{equation*}

\begin{definition}\label{deff} Let $\F(\lambda,n)$ denote the set of fillings $\sigma\::\:\lambda\rightarrow [n]$ satisfying
\begin{itemize}
\item $\sigma(u)\ne\sigma(v)$ whenever  $u$ and $v$ attack each other, and
\item $\sigma$ is weakly decreasing in rows, i.e., $\sigma(u)\ge\sigma(\rt(u))$.
\end{itemize}
\end{definition}

The {\it (Japanese) reading order} is the total order on the cells of $\lambda $
given by reading each column from top to bottom, and by considering the columns from left to right.

\begin{definition} An {\it
inversion} of $\sigma $ is a pair of attacking cells $(u,v)$ where  $\sigma (u)<\sigma (v)$ and $u$ precedes $v$ in the
reading order.  
\end{definition}

Here are two examples of inversions, where $a<b$:
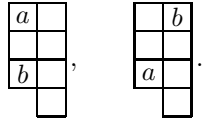
\begin{equation*}
\tableau{{a}&{}\\ {}&{}\\{b}&{}\\&{}} \, ,\qquad 
\tableau{{}&{b}\\ {}&{}\\{a}&{}\\&{}}\,.
\end{equation*}

\begin{definition} The {\em inversion statistic} on fillings $\sigma$, denoted $\inv(\sigma)$, is the number of inversions of $\sigma$. The {\em descent statistic}, denoted $\des(\sigma)$, is the number of cells $u=(i,j)$ with $j\ne 1$ and $\sigma(u)>\sigma(\rt(u))$.
\end{definition}

\begin{remark}\label{convfill} The inversion statistic is defined for arbitrary fillings in \cite{hhlcfm}, being used in the combinatorial formula for the Macdonald polynomials. In order to specialize the general definition to the one in this paper, one needs to restrict the entries of the fillings to $[n]$, replace each entry $i$ with $n+1-i$, and convert French diagrams to Japanese ones via the reflection in the line $y=-x$, as discussed above. 
\end{remark}

We are now ready to state a combinatorial formula for the Hall-Littlewood $Q$-polynomials in the variables $X=(x_1,\ldots,x_n)$ for a fixed $n$. This is quickly derived from the formula for the Macdonald's {\em integral form} symmetric polynomials $J_\lambda(X;q,t)$ \cite[Proposition 8.1]{hhlcfm} upon setting the variable $q$ to 0. In fact, as discussed in Remark \ref{convfill}, one also needs to set $x_i=0$ for $i>n$, and replace $x_i$ by $x_{n+1-i}$ for $i\in[n]$ in the cited formula (the Hall-Littlewood polynomials being symmetric in $X$). 

\begin{theorem}{\rm (}cf. \cite{hhlcfm}{\rm )}\label{hlqthm}
For any partition $\lambda$, we have
\begin{equation}\label{hlqform}
Q_{\lambda}(X;t) = \sum_{\sigma\in\mathcal{F}(\lambda,n)}
 t^{n(\lambda) - \inv (\sigma )}\,(1-t)^{\ell(\lambda)+\des(\sigma)}\: x^{{\rm ct}(\sigma) } \,.
\end{equation}
\end{theorem}

\begin{definition} We call $n(\lambda)-\inv(\sigma)$ the {\em complementary inversion statistic}, and denote it by $\cinv(\sigma)$.
\end{definition}

\begin{proposition} The statistic $\cinv(\sigma)$ counts the pairs of cells $(u,v)$ in the same column with $u$ below $v$, such that $\sigma(u)<\sigma(v)<\sigma(w)$, where $w$ is the cell directly to the left of $u$, if it exists (otherwise, the condition consists only of the first inequality).
\end{proposition}

\begin{proof}
Observe first that $n(\lambda)$ is the number of all pairs of cells $(u,v)$ in the same column with $u$ below $v$. We have $\sigma(v)\ne\sigma(u)$ and, if $w$ exists, then $\sigma(w)\ne\sigma(v)$, by the first condition in Definition \ref{deff}. The result follows simply by noting that, if $w$ exists, then we cannot simultaneously have $\sigma(w)<\sigma(v)$ and $\sigma(v)<\sigma(u)$, because the second condition in Definition \ref{deff} would be contradicted. 
\end{proof}

Here is an example of a configuration counted by the complementary inversion statistic, where $a<b<c$.
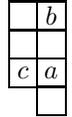
\[\tableau{{}&{b}\\ {}&{}\\{c}&{a}\\&{}}\,.\]

We conclude this section by recalling the relation between the Hall-Littlewood polynomials $P_\lambda(X;t)$ and $Q_\lambda(X;t)$. Assume that $\lambda$ has $m_i$ parts equal to $i$ for each $i$. As usual, for any positive integer $m$, we write
\[[m]_t!:=[m]_t[m-1]_t\ldots 1_t\,,\;\;\;\mbox{where }\,[k]_t:=\frac{1-t^k}{1-t}\,.\]
Then
\begin{equation}\label{pq}
P_\lambda(X;t)=\frac{1}{(1-t)^{\ell(\lambda)}[m_1]_t!\ldots [m_n]_t!}\:Q_\lambda(X;t)\,.
\end{equation}
Note that (\ref{hlqform}) makes the divisibility of $Q_\lambda(X;t)$ by $(1-t)^{\ell(\lambda)}$ obvious. 

\section{The compression phenomenon}

\subsection{Specializing Schwer's formula to type $A$}\label{specschwer} We now restrict ourselves to the root system of type $A_{n-1}$, fow which the Weyl group $W$ is the symmetric group $S_n$. Permutations $w\in S_n$ are written in one-line notation $w=w(1)\ldots w(n)$. 
We can identify the space $\h_\R^*$ with the quotient space 
$V:=\R^n/\R(1,\ldots,1)$,
where $\R(1,\ldots,1)$ denotes the subspace in $\R^n$ spanned 
by the vector $(1,\ldots,1)$.  
The action of the symmetric group $S_n$ on $V$ is obtained 
from the (left) $S_n$-action on $\R^n$ by permutation of coordinates.
Let $\varepsilon_1,\ldots,\varepsilon_n\in V$ 
be the images of the coordinate vectors in $\R^n$.
The root system $\Phi$ can be represented as 
$\Phi=\{\alpha_{ij}:=\varepsilon_i-\varepsilon_j \::\: i\ne j,\ 1\leq i,j\leq n\}$.
The simple roots are $\alpha_i=\alpha_{i,i+1}$, 
for $i=1,\ldots,n-1$.
The fundamental weights are $\omega_i = \varepsilon_1+\ldots +\varepsilon_i$, 
for $i=1,\ldots,n-1$. 
The weight lattice is $\Lambda=\Z^n/\Z(1,\ldots,1)$. A dominant weight $\lambda=\lambda_1\varepsilon_1+\ldots+\lambda_{n-1}\varepsilon_{n-1}$ is identified with the partition $(\lambda _{1}\geq \lambda _{2}\geq \ldots \geq \lambda _{n-1}\geq\lambda_n=0)$ of length at most $n-1$. We fix such a partition $\lambda$ for the remainder of this paper.

For simplicity, we use the same notation $(i,j)$ with $i<j$ for the root $\alpha_{ij}$ and the reflection $s_{\alpha_{ij}}$, which is the transposition of $i$ and $j$.  
We proved in \cite[Corollary 15.4]{lapawg} that, for any $k=1,\ldots,n-1$,
we have the following $\omega_k$-chain, denoted by $\Gamma(k)$:
\begin{equation}\label{omegakchain}\begin{array}{lllll}
(&\!\!\!\!(1,n),&(1,n-1),&\ldots,&(1,k+1)\,,\\
&\!\!\!\!(2,n),&(2,n-1),&\ldots,&(2,k+1)\,,\\
&&&\ldots\\
&\!\!\!\!(k,n),&(k,n-1),&\ldots,&(k,k+1)\,\,)\,.
\end{array}\end{equation}
Hence, we can construct a $\lambda$-chain as a concatenation $\Gamma:=\Gamma_{\lambda_1}\ldots\Gamma_1$, where $\Gamma_j=\Gamma(\lambda'_j)$. This $\lambda$-chain is fixed for the remainder of this paper. Thus, we can replace the notation ${\mathcal A}(\Gamma)$ with ${\mathcal A}(\lambda)$.

\begin{example}\label{ex21} {\rm Consider $n=4$ and $\lambda =(2,1,0,0)$, for which we have the following $\lambda$-chain (the underlined pairs are only relevant in Example \ref{ex21c} below):
\begin{equation}\label{exlchain}\Gamma=\Gamma_2\Gamma_1=(\underline{(1,4)},(1,3),(1,2)\:|\:(1,4),\underline{(1,3)},(2,4),\underline{(2,3)})\,.\end{equation}
We represent the Young diagram of $\lambda$ inside a broken $4\times 2$ rectangle, as below. In this way, a transpositions $(i,j)$ in $\Gamma$ can be viewed as swapping entries in the two parts of each column (in rows $i$ and $j$, where the row numbers are also indicated below). 
\begin{equation*}
 \begin{array}{l} \tableau{{1}&{1}\\ &{2}}\\ \\
\tableau{{2}\\ {3}&{3}\\ {4}&{4}} \end{array}
\end{equation*}}
\end{example}

Given the $\lambda$-chain $\Gamma$ above, in Section \ref{alcovewalks} we considered subsets $J=\{ j_1<\ldots< j_s\}$ of $[m]$, where in the present case $m=\sum_{i=1}^{\lambda_1}\lambda_i'(n-\lambda_i')$. Instead of $J$, it is now convenient to use the subsequence of $\Gamma$ indexed by the positions in $J$. This is viewed as a concatenation with distinguished factors $T=T_{\lambda_1}\ldots T_1$ induced by the factorization of $\Gamma$ as $\Gamma_{\lambda_1}\ldots\Gamma_1$. More precisely, $T_j$ is the subsequence of $\Gamma_j$ indexed by the following positions in $\Gamma$:
\[J\cap\left(\sum_{i=j+1}^{\lambda_1}\lambda_i'(n-\lambda_i'),\:\sum_{i=j}^{\lambda_1}\lambda_i'(n-\lambda_i')\right]\,.\]

All the notions defined in terms of $J$ are now redefined in terms of $T$. As such, from now on we will write  $\phi(T)$, $\mu(T)$, and $|T|$, the latter being the size of $T$. If $J$ is a $w$-admissible subset for some $w$ in $S_n^\lambda$, we will also call the corresponding $T$ a {\em $w$-admissible sequence}, and $(w,T)$ an admissible pair. We will use the notation ${\mathcal A}(\Gamma)$ and ${\mathcal A}(\lambda)$ accordingly.  

We denote by $wT_{\lambda_1}\ldots T_{j}$ the permutation obtained from $w$ via right multiplication by the transpositions in $T_{\lambda_1},\ldots, T_{j}$, considered from left to right. This agrees with the above convention of using pairs to denote both roots and the corresponding reflections. As such, $\phi(J)$ in (\ref{defphimu}) can now be written simply $T$. 

\begin{remark}\label{crit} Let us introduce the following order on pairs $(a,b)$ with $a<b$:
\[(a,b)\prec(c,d)\;\;\;\mbox{if $(a<c)$ or ($a=c$ and $b>d$)}\,.\]
In the present setup, an admissible pair can be defined as a pair $(w,T)$, where $w$ is a permutation in $S_n^\lambda$, $T=T_{\lambda_1}\ldots T_1$, and each $T_j$ is a sequence of pairs $((a_1,b_1),\ldots,(a_p,b_p))$ such that
\begin{itemize}
\item $1\le a_i\le\lambda_j'< b_i\le n$ for $i=1,\ldots,p$;
\item $(a_i,b_i)\prec(a_{i+1},b_{i+1})$  for $i=1,\ldots,p-1$;
\item $T_j$ is the sequence of labels on a chain in the reverse Bruhat graph on $S_n$ which starts at $wT_{\lambda_1}\ldots T_{j+1}$. 
\end{itemize}
\end{remark}

\begin{example}\label{ex21c}{\rm We continue Example \ref{ex21}, by picking the admissible pair $(w,J)$ with $w=4312\in S_4^\lambda$ and $J=\{1,5,7\}$ (see the underlined positions in (\ref{exlchain})). Thus, we have
\[T=T_2T_1=((1,4)\:|\:(1,3),(2,3))\,.\]
The corresponding decreasing chain in Bruhat order is the following, where the swapped entries are shown in bold (we represent permutations as broken columns, as discussed in Example \ref{ex21}):
\[w=\begin{array}{l}\tableau{{{\mathbf 4}}} \\ \\ \tableau{{3}\\{1}\\{{\mathbf 2}}} \end{array} \:>\:\begin{array}{l}\tableau{{2}} \\ \\ \tableau{{3}\\{1}\\{4}} \end{array}\:|\: \begin{array}{l}\tableau{{{\mathbf 2}}\\{3}} \\ \\ \tableau{{{\mathbf 1}}\\{4}} \end{array}\:>\: \begin{array}{l}\tableau{{1}\\{{\mathbf 3}}} \\ \\ \tableau{{{\mathbf 2}}\\{4}}\end{array}\:>\: \begin{array}{l}\tableau{{1}\\{2}}\\ \\ \tableau{{3}\\{4}}\end{array} \,.\]
}
\end{example}

\subsection{The filling map} Given a pair $(w,T)$, not necessarily admissible, we consider the permutations
\[\pi_j=\pi_j(w,T):=wT_{\lambda_1}T_{\lambda_1-1}\ldots T_{j+1}\,,\]
for $j=1,\ldots,\lambda_1$. In particular, $\pi_{\lambda_1}=w$. 

\begin{definition}\label{deffill}
The {\em filling map} is the map $f$ from pairs $(w,T)$, not necessarily admissible, to fillings $\sigma=f(w,T)$ of the shape $\lambda$, defined by
\begin{equation}\label{defswt}\sigma(i,j):=\pi_j(i)\,.\end{equation}
In other words, the $j$-th column of the filling $\sigma$ consists of the first $\lambda_j'$ entries of the permutation $\pi_j$.
\end{definition}

\begin{example} {\rm Given $(w,T)$ as in Example \ref{ex21c}, we have
\[f(w,T)=\tableau{{4}&{2}\\&{3}}\,.\]}
\end{example}

\begin{proposition}\label{weightmon} Given a permutation $w$ and any subsequence $T$ of $\Gamma$ (not necessarily $w$-admissible), we have ${\rm ct}(f(w,T))=w(\mu(T))$. In particular, $w(\mu(T))$ only depends on $f(w,T)$. \end{proposition}

This proposition is proved in Section \ref{pfweight}, where it is shown how to recover the affine information from the fillings.

\begin{proposition}\label{surjmap}
We have $f(\mathcal{A}(\lambda))\subseteq\F(\lambda,n)$. If the partition $\lambda$ has $n-1$ distinct non-zero parts, then the map $f\::\:\mathcal{A}(\lambda)\rightarrow\F(\lambda,n)$ is surjective. 
\end{proposition}

\begin{proof}
 Let $u=(i,j)$ be a cell of $\lambda$, and $\sigma=f(w,T)$. We check that $\sigma$ satisfies the two conditions in Definition \ref{deff}. If $j\ne 1$, then $\sigma(u)\ge\sigma(\rt(u))$ by the decreasing chain condition in (\ref{decch}). Now let $v=(k,j)$ with $k>i$. We clearly have $\sigma(u)\ne\sigma(v)$ because $\sigma(u)=\pi_j(i)$ and $\sigma(v)=\pi_j(k)$. For the same reason, if $\sigma(u)=\sigma(\rt(u))$, then $\sigma(v)\ne\sigma(\rt(u))$. Otherwise, consider the subchain of the Bruhat chain corresponding to $(w,T)$ which starts at $\pi_j$ and ends at $\pi_{j-1}$. There is a permutation $\pi$ in this subchain such that $\sigma(v)=\pi(k)$ and $\sigma(\rt(u))$ is the entry in some position greater than $k$, to be swapped with position $i$; this follows from the structure of the segment $\Gamma_j$ (see (\ref{omegakchain})) in the $\lambda$-chain $\Gamma$. Thus, we have $\sigma(v)\ne\sigma(\rt(u))$ once again. We conclude that $\sigma\in\F(\lambda,n)$. 

Now consider $\lambda$ with $n-1$ distinct non-zero parts, and $\sigma\in\F(\lambda,n)$. We construct an increasing chain in the Bruhat order on $S_n$, as follows. Let $\pi_1$ be some permutation such that $\pi_1(i)=\sigma(i,1)$ for $1\le i\le\lambda_1'$. For each $i$ from $\lambda_2'$ down to 1, if $\sigma(i,2)\ne\sigma(i,1)$, then swap the entry in position $i$ with the entry $\sigma(i,2)$; the latter is always found in a position greater than $\lambda_2'$ because $\sigma\in\F(\lambda,n)$. The result is a permutation $\pi_2$ whose first $\lambda_2'$ entries form the second column of $\sigma$. Continue in this way, by moving to columns $3,\ldots,\lambda_1$. Then set $w:=\pi_{\lambda_1}$. The constructed Bruhat chain determines an admissible pair $(w,T)$ mapped to $\sigma$ by $f$. 
\end{proof}

Based on Proposition \ref{surjmap}, from now on we consider the filling map as a map $f\::\:{\mathcal A}(\lambda)\rightarrow\F(\lambda,n)$. 

\subsection{Compressing Schwer's formula} In this section we assume that the partition $\lambda$  has $n-1$ distinct non-zero parts, i.e., corresponds to a regular weight in the root system $A_{n-1}$. In this case $S_n^\lambda=S_n$, and thus the pairs $(w,J)$ in ${\mathcal A}(\lambda)$ are only subject to the decreasing chain condition in (\ref{decch}); this fact is implicitly used in the proof of the theorem below. 

We start by recalling Ram's version of Schwer's formula (\ref{hlpform}) and the HHL formula (\ref{hlqform}):
\begin{align*}
&P_{\lambda}(X;t)=\sum_{(w,T)\in{\mathcal A}(\Gamma)}t^{\frac{1}{2}(\ell(w)+\ell(wT)-|T|)}\,(1-t)^{|T|}\,x^{w(\mu(T))}\,,\\
&P_{\lambda}(X;t) = \sum_{\sigma\in\mathcal{F}(\lambda,n)}
 t^{\cinv (\sigma )}\,(1-t)^{\des(\sigma)}\: x^{{\rm ct}(\sigma) } \,.
\end{align*}

We will now describe the way in which the second formula can be obtained by compressing the first one.

\begin{theorem}\label{mainthm}
Given $\lambda$ with $n-1$ distinct non-zero parts and any $\sigma \in{\mathcal F}(\lambda,n)$, we have $f^{-1}(\sigma)\ne\emptyset$ and $x^{w(\mu(T))}=x^{{\rm ct}(\sigma)}$ for all $(w,T)\in f^{-1}(\sigma)$. Furthermore, we have
\begin{equation}\label{compress}\sum_{(w,T)\in f^{-1}(\sigma)}t^{\frac{1}{2}(\ell(w)+\ell(wT)-|T|)}\,(1-t)^{|T|}= t^{\cinv (\sigma )}\, (1-t)^{\des(\sigma)}\,.\end{equation}
\end{theorem}

The first statement in the theorem is just the content of Propositions \ref{weightmon} and \ref{surjmap}. The compression formula (\ref{compress}) will be proved Section \ref{pfmain}.

The case of an arbitrary partition $\lambda$ (with at most $n-1$ parts) is considered in the Appendix (Section \ref{al}). 

In order to measure the compression phenomenon, we define the {\em compression factor} $c(\lambda)$ as the ratio of the number of terms in Ram's version of Schwer's formula for $\lambda$ and the number of terms $t(\lambda)$ in the HHL formula above. We list below some examples, which show that the compression factor increases sharply with the number of variables $n$. 

\begin{equation*}
\begin{tabular}{|c|c|r|r|}
\hline
$\lambda$ &  $n$ & $t(\lambda)$ & $c(\lambda)$ 
 \\ \hline
(4, 2, 1) & $4$ & 366 & 2.9 
 \\ \hline
(4, 2, 1) & $5$ & 1,869  & 9.0 
 \\ \hline
(4, 2, 1) & $6$ & 6,832  &  31.3
 \\ \hline
(4, 3, 2, 1) & $5$ & 8,896  &  4.1
 \\ \hline
(4, 3, 2, 1) & $6$ & 75,960  & 17.7
 \\ \hline
\end{tabular}
\end{equation*}

\section{The proof of Proposition \ref{weightmon}}\label{pfweight}

The weight/partition $\lambda$ is assumed to be an arbitrary one in this section. Recall the $\lambda$-chain $\Gamma$ in Section \ref{specschwer}. Let us write $\Gamma=(\beta_1,\ldots,\beta_m)$, as in Section \ref{alcovewalks}. As such, we recall the hyperplanes $H_{\beta_k,l_k}$ and the corresponding affine reflections $\widehat{r}_k=s_{\beta_k,l_k}$. If $\beta_k=(a,b)$ falls in the segment $\Gamma_p$ of $\Gamma$ (upon the factorization $\Gamma=\Gamma_{\lambda_1}\ldots\Gamma_1$ of the latter), then it is not hard to see that
\[l_k=|\{i\::\:p\le i\le\lambda_1,\:\lambda_i'\ge a\}|\,.\]
An affine reflection $s_{(a,b),l}$ acts on our vector space $V$ by
\begin{equation}\label{affact}s_{(a,b),l}(\mu_1,\ldots,\mu_a,\ldots,\mu_b,\ldots,\mu_n):=(\mu_1,\ldots,\mu_b+l,\ldots,\mu_a-l,\ldots,\mu_n)\,.\end{equation}

Now fix a permutation $w$ in $S_n$ and a subset $J=\{j_1<\ldots<j_s\}$ of $[m]$ (not necessarily $w$-admissible).  Let $\Pi$ be the alcove path corresponding to $\Gamma$, and define the alcove walk $\Omega$ as in Section \ref{alcovewalks}, by
\[\Omega:=\phi_{j_1}\ldots \phi_{j_s} (w(\Pi))\,.\]
Given $k$ in $[m]$, let $i=i(k)$ be the largest index in $[s]$ for which $j_i<k$. Let $\gamma_k:=wr_{j_1}\ldots r_{j_i}(\beta_k)$, and let $H_{\gamma_k,m_k}$ be the hyperplane containing the face $F_k$ of $\Omega$ (cf. Definition \ref{defalcwalk}). In other words
\[H_{\gamma_k,m_k}=w\widehat{r}_{j_1}\ldots \widehat{r}_{j_i}(H_{\beta_k,l_k})\,.\]
Our first goal is to describe $m_k$ purely in terms of the filling associated to $(w,J)$.

Let $\widehat{t}_k$ be the affine reflection in the hyperplane $H_{\gamma_k,m_k}$. Note that
\[\widehat{t}_k=w\widehat{r}_{j_1}\ldots \widehat{r}_{j_i}\widehat{r}_k\widehat{r}_{j_i}\ldots \widehat{r}_{j_1}w^{-1}\,.\]
Thus, we can see that
\[w\widehat{r}_{j_1}\ldots \widehat{r}_{j_i}=\widehat{t}_{j_i}\ldots \widehat{t}_{j_1}w\,.\]

Let $T=((a_1,b_1),\ldots,(a_s,b_s))$ be the subsequence of $\Gamma$ indexed by the positions in $J$, cf. Section \ref{specschwer}. Let $T^i$ be the initial segment of $T$ with length $i$, let $w_i:=wT^i$, and $\sigma_i:=f(w,T^i)$. In particular, $\sigma_0$ is the filling with all entries in row $i$ equal to $w(i)$, and  $\sigma:=\sigma_s=f(w,T)$. The columns of a filling of $\lambda$ are numbered, as usual, from left to right by $\lambda_1$ to 1. Note that, if $\beta_{j_{i+1}}=(a_{i+1},b_{i+1})$ falls in the segment $\Gamma_p$ of $\Gamma$, then $\sigma_{i+1}$ is obtained from $\sigma_i$ by replacing the entry $w_i(a_{i+1})$ with $w_i(b_{i+1})$ in the columns $p-1,\ldots,1$ (and the row $a_{i+1}$) of $\sigma_i$. 

Given a fixed $k$, let $\beta_k=(a,b)$, $c:=w_i(a)$, and $d:=w_i(b)$, where $i=i(k)$ is defined as above. Then $\gamma:=\gamma_k=(c,d)$, where we might have $c>d$. Let $\Gamma_q$ be the segment of $\Gamma$ where $\beta_k$ falls. Given a filling $\phi$, we denote by $\phi[p]$ and $\phi(p,q]$ the parts of $\phi$ consisting of the columns $\lambda_1,\lambda_1-1,\ldots,p$ and $p-1,p-2,\ldots,q$, respectively. We use the notation $N_{e}(\phi)$ to denote the number of entries $e$ in the filling $\phi$. 

\begin{proposition}\label{level} With the above notation, we have
\[m_k=\langle{\rm ct}(\sigma[q]),\gamma^\vee\rangle=N_c(\sigma[q])-N_d(\sigma[q])\,.\]
\end{proposition}

\begin{proof}  We apply induction on $i$, which starts at $i=0$. We will now proceed from $j_1<\ldots<j_{i}<k$, where $i=s$ or $k\le j_{i+1}$, to $j_1<\ldots<j_{i+1}<k$, and we will freely use the notation above. Let
\[\beta_{j_{i+1}}=(a',b')\,,\;\; c':=w_i(a')\,,\;\;d':=w_i(b')\,.\]
Let $\Gamma_p$ be the segment of $\Gamma$ where $\beta_{j_{i+1}}$ falls, where $p\ge q$. 

We need to compute
\[w\widehat{r}_{j_1}\ldots \widehat{r}_{j_{i+1}}(H_{\beta_k,l_k})=\widehat{t}_{j_{i+1}}\ldots \widehat{t}_{j_1}w(H_{\beta_k,l_k})=\widehat{t}_{j_{i+1}}(H_{\gamma,m})\,,\]
where $m=\langle{\rm ct}(\sigma_i[q]),\gamma^\vee\rangle$, by induction. Note that $\gamma':=\gamma_{j_{i+1}}=(c',d')$, and $\widehat{t}_{j_{i+1}}=s_{\gamma',m'}$, where $m'=\langle{\rm ct}(\sigma_i[p]),(\gamma')^\vee\rangle$, by induction. We will use the following formula:
\[s_{\gamma',m'}(H_{\gamma,m})=H_{s_{\gamma'}(\gamma),m-m'\langle\gamma',\gamma^\vee\rangle}\,.\]
Thus, the proof is reduced to showing that
\[m-m'\langle\gamma',\gamma^\vee\rangle=\langle{\rm ct}(\sigma_{i+1}[q]),s_{\gamma'}(\gamma^\vee)\rangle\,.\] 
An easy calculation, based on the above information, shows that the latter equality is non-trivial only if $p>q$, in which case it is equivalent to
\begin{equation}\label{indlev}\langle{\rm ct}(\sigma_{i+1}(p,q])-{\rm ct}(\sigma_{i}(p,q]),\gamma^\vee\rangle=\langle\gamma',\gamma^\vee\rangle\,\langle{\rm ct}(\sigma_{i+1}(p,q]),(\gamma')^\vee\rangle\,.\end{equation}
This equality is a consequence of the fact that
\[{\rm ct}(\sigma_{i+1}(p,q])=s_{\gamma'}({\rm ct}(\sigma_{i}(p,q]))\,,\]
which follows from the construction of $\sigma_{i+1}$ from $\sigma_i$ explained above.
\end{proof}

\begin{proof}[Proof of Proposition {\rm \ref{weightmon}}] We apply induction on the size of $T$, using freely the notation above. We prove the statement for $T=((a_1,b_1),\ldots,(a_{s+1},b_{s+1}))$, assuming it holds for $T^s=((a_1,b_1),\ldots,(a_{s},b_{s}))$. We have
\[w(\mu(T)=w\widehat{r}_{j_1}\ldots \widehat{r}_{j_{s+1}}(\lambda)=\widehat{t}_{j_{s+1}}\ldots \widehat{t}_{j_1}w(\lambda)=\widehat{t}_{j_{s+1}}({\rm ct}(\sigma_s))\,,\]
by induction. Let $c:=w_s(a_{s+1})$, $d:=w_s(b_{s+1})$, and assume that $(a_{s+1},b_{s+1})$ falls in the segment of $\Gamma$ corresponding to column $p$ of $\lambda$. Then $\widehat{t}_{j_{s+1}}=s_{(c,d),m}$, where 
\[m=N_{c}(\sigma_s[p])-N_{d}(\sigma_s[p])\,,\]
by Proposition \ref{level}. Thus, the proof is reduced to
\[s_{(c,d),m}(N_1(\sigma_{s}),\ldots,N_{n}(\sigma_{s}))=(N_{1}(\sigma_{s+1}),\ldots,N_{n}(\sigma_{s+1}))\,.\]
This follows easily from  (\ref{affact}) and the following equalities:
\begin{align*}&N_{c}(\sigma_{s+1})-N_{c}(\sigma_s[p])=|\{j\::\:\lambda_j'\ge b_{s+1}\}|=N_{d}(\sigma_{s})-N_{d}(\sigma_s[p])\,,\\
&N_{d}(\sigma_{s+1})-N_{d}(\sigma_s[p])=p-1=N_{c}(\sigma_{s})-N_{c}(\sigma_s[p])\,.\end{align*}
\end{proof}

\section{The proof of Theorem \ref{mainthm}}\label{pfmain}

The weight/partition $\lambda$ is assumed to be an arbitrary one in this section, unless otherwise specified. We start with some notation related to sequences of positive integers. Given such a sequence $w$, we write $w[i,j]$ for the subsequence $w(i)w(i+1)\ldots w(j)$. We use the notation $N_a(w)$ and $N_{ab}(w)$ for the number of entries $w(i)$ with $w(i)<a$ and $a<w(i)<b$, respectively.  Given a sequence of transpositions $T=((a_1,b_1),\ldots,(a_p,b_p))$ and a permutation $w$, we define
\begin{equation}\label{defnwt}N(w,T):=\sum_{i=1}^p N_{c_id_i}(w_i[a_i,b_i])\,,\end{equation}
where 
\[w_i=w(a_1,b_1)\ldots(a_i,b_i)\,,\;\;\;c_i:=\min(w_i(a_i),w_i(b_i))\,,\;\;\;d_i:=\max(w_i(a_i),w_i(b_i))\,.\]
We extend the notion of length $\ell(\,\cdot\,)$ for permutations to any sequence $w$ by counting inversions, which are defined in the same way as for permutations. All the formulas below and the relevant notions are tacitly extended, in the obvious way, to permutations of some fixed set of $n$ positive integers.

\begin{lemma}\label{split}
Consider $(w,T)\in{\mathcal A}(\Gamma)$ with $T$ written as a concatenation $S_1\ldots S_p$. Let $w_i:=wS_1\ldots S_i$, so $w_0=w$. Then
\[\frac{1}{2}(\ell(w)+\ell(wT)-|T|)=\frac{1}{2}(\ell(wS_1\ldots S_{p-1})+\ell(wT)-|S_p|)+\sum_{i=1}^{p-1}N(w_{i-1},S_i)\,.\]
\end{lemma}

\begin{proof}It is enough to prove the lemma for $p=2$, since the function $N(u,S)$ has the obvious additivity property. For $p=2$, we simply note that
\begin{equation}\label{tailb}\frac{1}{2}(\ell(w)-\ell(wS_1)-|S_1|)=N(w,S_1)\,.\end{equation}
\end{proof}

Given the $\omega_k$-chain $\Gamma(k)$ in (\ref{omegakchain}) and $0\le p\le n-k$, let $\Gamma(k,p)$ denote the segment of $\Gamma(k)$ obtained by removing its first $p$ entries. We define ${\mathcal A}(\Gamma(k,p))$ like in (\ref{decch}), except that $w\in W$. 

\begin{proposition}\label{inv2} Fix $k,p$ such that $1\le k\le n-p\le n$. Consider a permutation $w$ in $S_n$ given as a concatenation $w_1w_2w_3$, where $w_1$ and $w_3$ are sequences of size $k$ and $p$, respectively.  We have
\begin{equation}\label{comp1}\sum_{T\::\:(w,T)\in{\mathcal A}(\Gamma(k,p))}t^{\frac{1}{2}(\ell(w)+\ell(wT)-|T|)}\,(1-t)^{|T|}=t^{\ell(w_1)+\ell(w_2w_3)+N_{w(1)}(w_3)}\,.\end{equation}
In particular, if $p=0$, and thus $w=w_1w_2$, the above sum is just $t^{\ell(w_1)+\ell(w_2)}$.
\end{proposition}

\begin{proof} Let us denote the sum in (\ref{comp1}) by $S(w,k,p)$. We prove (\ref{comp1}) for all $n$ by increasing induction on $k$ and, for a fixed $k$, by decreasing induction on $p$. Induction starts at $k=1$ and $p=n-1$, when $S(w,k,p)$ consists of a single term, corresponding to $T=\emptyset$. Now let $a:=w(1)$ and $w_1=aw_1'$. If $p=n-k$, which means that $w_2$ is empty, then 
\[S(w,k,p)=t^{a-1}S(w_1'w_3,k-1,0)=t^{a-1}t^{\ell(w_1')+\ell(w_3)}\,,\]
by induction. But $\ell(w_1')+a-1=\ell(w_1)+N_a(w_3)$, which concludes the induction step. 

Now let us prove (\ref{comp1}) for $p<n-k$, assuming that it holds for $p+1$. As before, let $w_1=aw_1'$; also let $b:=w(n-p)$ and $w_2=w_2'b$. The case $a<b$ is straightforward, so let us assume $a>b$. The sum $S(w,k,p)$ splits into two sums, depending on $(1,n-p)\not\in T$ or $(1,n-p)\in T$. By induction, the first sum is 
\[S(w,k,p+1)=t^{\ell(w_1)+\ell(w_2w_3)+N_a(w_3)+1}\,.\]
By Lemma \ref{split}, the second sum is
\[(1-t)\,t^{N_{ba}(w_1')+N_{ba}(w_2')}S(bw_1'w_2'aw_3,k,p+1)\,.\]
By induction, 
\[S(bw_1'w_2'aw_3,k,p+1)=t^{\ell(bw_1')+\ell(w_2'aw_3)+N_b(w_3)}\,.\]
But we have
\[N_{ba}(w_1')+\ell(bw_1')=\ell(w_1)\,,\;\;N_{ba}(w_2')+\ell(w_2'aw_3)=\ell(w_2w_3)+N_{ba}(w_3)\,,\;\;N_{ba}(w_3)+N_b(w_3)=N_a(w_3).\]
Assembling the above information, we have
\[S(w,k,p)=t^{\ell(w_1)+\ell(w_2w_3)+N_a(w_3)+1}+(1-t)\,t^{\ell(w_1)+\ell(w_2w_3)+N_a(w_3)}=t^{\ell(w_1)+\ell(w_2w_3)+N_a(w_3)}\,.\]
This completes the induction step.
\end{proof}

Let us denote by $\rev(S)$ the reverse of the sequence $S$. For simplicity, we write $\Gamma^r(k)$ for $\rev(\Gamma(k))$. We also consider the segment of $\Gamma^r(k)$ with the first $p$ entries removed, which we denote by $\Gamma^r(k,p)$. Given a segment $\Gamma'$ of $\rev(\Gamma)$ (where $\Gamma$ is our fixed $\lambda$-chain), in particular $\Gamma'=\Gamma^r(k,p)$, we define ${\mathcal A}^r(\Gamma')$ like in (\ref{decch}) except that we impose an increasing chain condition and $w\in W$. 

\begin{proposition}\label{p2cols0}
Consider a permutation $w$ in $S_n$ and a number $b\in[n]$ such that $b\ge a:=w(1)$; also consider an integer $p$ with $0\le p<w^{-1}(b)-1$. Then we have
\begin{equation}\label{sum2cols0}
\stacksum{T\::\:(w,T)\in{\mathcal A}^r(\Gamma^r(1,p))}{wT(1)=b} t^{N(w,T)} (1-t)^{|T|}=t^{N_{ab}(w[2,p+1])} (1-t)^{1-\delta_{ab}}\,;
\end{equation}
here $\delta_{ab}$ is the Kronecker delta.
\end{proposition}

\begin{proof}
Assume that $a<b$. We use decreasing induction on $p$. The base case for $p=w^{-1}(b)-2$ is straightforward. Let us now prove the statement for $p$ assuming it for $p+1$. If $c:=w(p+2)\not\in(a,b)$, the induction step is straightforward, so assume the contrary. In this case, the sum in (\ref{sum2cols0}), denoted by $S(w,p)$, splits into two sums, depending on $(1,p+2)\not\in T$ and $(1,p+2)\in T$. By induction, the first sum is
\[S(w,p+1)=t^{N_{ab}(w[2,p+1])+1}(1-t)\,.\]
The second sum is
\begin{align*}t^{N_{ac}(w[2,p+1])}(1-t)\,S(w(1,p+2),p+1)&=t^{N_{ac}(w[2,p+1])}(1-t)\,t^{N_{cb}(w[2,p+1])}(1-t)\\&=t^{N_{ab}(w[2,p+1])}(1-t)^2\,,\end{align*}
where the first equality is obtained by induction. The desired result easily follows by adding the two sums into which $S(w,p)$ splits. 
\end{proof}

\begin{proposition}\label{twocol} Consider two sequences $C_1\le C_2$ of size $k$ and entries in $[n]$, where $\le$ denotes the componentwise order. Also consider a permutation $w$ in $S_n$ such that $w[1,k]=C_1$. Then we have
\begin{equation}\label{sum2cols}
\stacksum{T\::\:(w,T)\in{\mathcal A}^r(\Gamma^r(k))}{wT[1,k]=C_2} t^{N(w,T)} (1-t)^{|T|}=t^{\ell(C_2)-\ell(C_1)+\cinv(C_2C_1)}(1-t)^{\des(C_2C_1)}\,;
\end{equation}
here $C_2C_1$ denotes the two-column filling with left column $C_2$ and right column $C_1$.
\end{proposition}

\begin{proof}
Let us split $\Gamma^r(k)$ as $\Gamma_k^r(k)\ldots\Gamma_1^r(k)$ by the rows in (\ref{omegakchain}), that is,
\[\Gamma_i^r(k)=(\,(i,k+1),\:(i,k+2),\:\ldots ,\:(i,n)\,)\,.\]
This splitting induces one for the subsequence $T$ of $\Gamma^r(k)$ indexing the sum in (\ref{sum2cols}), namely $T=T_k\ldots T_1$. For $i=0,1,\ldots,k$, define $w_i:=wT_k T_{k-1}\ldots T_{i+1}$, so $w_k=w$. The sum in (\ref{sum2cols}) can be written as a $k$-fold sum in the following way:
\begin{equation}\label{kfold}\stacksum{T_1\::\:(w_1,T_1)\in{\mathcal A}^r(\Gamma_1^r(k))}{w_0(1)=C_2(1)} t^{N(w_1,T_1)} (1-t)^{|T_1|}\;\ldots  \stacksum{T_k\::\:(w_k,T_k)\in{\mathcal A}^r(\Gamma_k^r(k))}{w_{k-1}(k)=C_2(k)} t^{N(w_k,T_k)} (1-t)^{|T_k|}\,.\end{equation}
By Proposition \ref{p2cols0}, we have
\[\stacksum{T_i\::\:(w_i,T_i)\in{\mathcal A}^r(\Gamma_i^r(k))}{w_{i-1}(i)=C_2(i)} t^{N(w_i,T_i)} (1-t)^{|T_i|}=t^{N_{C_1(i),C_2(i)}(C_2[i+1,k])} (1-t)^{1-\delta_{C_1(i),C_2(i)}}\,.\]
We can see that the above sum does not depend on the permutation $w_{i}$, but only on $C_1(i)$ and $C_2[i,k]$. Therefore, the $k$-fold sum (\ref{kfold}) is a product of $k$ factors, and evaluates to $t^e(1-t)^{\des(C_2C_1)}$, where
\[e=\sum_{i=1}^{k-1} N_{C_1(i),C_2(i)}(C_2[i+1,k])\,.\]
Clearly, $e$ is just the number of inversions $(i,j)$ in $C_2$ (meaning that $i<j$ and $C_2(i)>C_2(j)$) for which $C_1(i)<C_2(j)$. If $(i,j)$ is an inversion in $C_2$ not satisfying the previous condition, then $C_1(i)> C_2(j)$ (by the first condition in Definition \ref{deff}), and thus $(i,j)$ is an inversion in $C_1$ (by the second condition in Definition \ref{deff}). Moreover, the only inversions of $C_1$ which do not arise in this way are those counted by the statistic $\cinv(C_2C_1)$, so $e=\ell(C_2)-(\ell(C_1)-\cinv(C_2C_1))$.
\end{proof}

 We can define a filling map $f^r\::\:{\mathcal A}^r(\rev(\Gamma))\rightarrow\F(\lambda,n)$ in a similar way to Definition \ref{deffill}, namely
\[f^r(w,T):=f(wT,\rev(T))\,.\]

\begin{proposition}\label{goleft} Consider a filling $\sigma\in\F(\lambda,n)$, whose rightmost column is $C:=(\sigma(1,1),\ldots,\sigma(\lambda_1',1))$. Let $w$ be a permutation in $S_n$ satisfying $w[1,\lambda_1']=C$. We have
\begin{equation}\label{compress0}\sum_{T\::\:(w,T)\in (f^r)^{-1}(\sigma)}t^{N(w,T)}\,(1-t)^{|T|}= t^{\cinv (\sigma )-\ell(C)}\, (1-t)^{\des(\sigma)}\,.\end{equation}
\end{proposition}

\begin{proof}
 The splitting $\rev(\Gamma)=\Gamma_1^r\ldots \Gamma_{\lambda_1}^r$, where $\Gamma$ is our fixed $\lambda$-chain and $\Gamma_i^r:=\rev(\Gamma_i)$, induces a splitting $T=T_1\ldots T_{\lambda_1}$ of any $T$ for which $(w,T)\in{\mathcal A}^r(\rev(\Gamma))$, cf. Section \ref{specschwer}. Let $m:=\lambda_1$ be the number of columns of $\lambda$, and let $C=C_1,\ldots,C_m$ be the columns of $\sigma$, of lengths $c_1:=\lambda_1',\ldots,c_m:=\lambda_m'$. Let $C_i':=C_i[1,c_{i+1}]$, for $i=1,\ldots,m-1$. Note first that the segment $T_1$ of a sequence $T$ in (\ref{compress0}) is empty. For $i=1,\ldots,m$, define $w_i:=wT_1T_2\ldots T_{i}$, so $w_1=w$. The sum in (\ref{compress0}) can be written as an $(m-1)$-fold sum in the following way:
\begin{equation}\label{mfold}\stacksum{T_{m}\::\:(w_{m-1},T_{m})\in{\mathcal A}^r(\Gamma_{m}^r)}{w_{m}[1,c_m]=C_m} t^{N(w_{m-1},T_m)} (1-t)^{|T_m|}\;\ldots  \stacksum{T_2\::\:(w_1,T_2)\in{\mathcal A}^r(\Gamma_2^r)}{w_{2}[1,c_2]=C_2} t^{N(w_1,T_2)} (1-t)^{|T_2|}\,.\end{equation}
By Proposition \ref{twocol}, we have
\[\stacksum{T_{i}\::\:(w_{i-1},T_{i})\in{\mathcal A}^r(\Gamma_{i}^r)}{w_{i}[1,c_i]=C_i} t^{N(w_{i-1},T_i)} (1-t)^{|T_i|}=t^{\ell(C_i)-\ell(C_{i-1}')+\cinv(C_iC_{i-1}')}(1-t)^{\des(C_iC_{i-1})}\,.\]
We can see that the above sum does not depend on the permutation $w_{i-1}$, but only on $C_{i-1}$ and $C_i$. Therefore, the $(m-1)$-fold sum (\ref{mfold}) is a product of $m-1$ factors, and evaluates to $t^e(1-t)^{\des(\sigma)}$, where
\[e=\sum_{i=2}^m \ell(C_i)-\ell(C_{i-1}')+\cinv(C_iC_{i-1}')\,.\]
Since $\ell(C_m)=\cinv(C_m)$, we have
\begin{align*}e+\ell(C)&=\cinv(C_m)+\sum_{i=2}^{m} \ell(C_{i-1})-\ell(C_{i-1}')+\cinv(C_iC_{i-1}')=\\&=\cinv(C_m)+\sum_{i=2}^{m} \cinv(C_iC_{i-1})=\cinv(\sigma)\,.\end{align*}
\end{proof}

\begin{proof}[Proof of Theorem {\rm \ref{mainthm}}] At this point we assume that $\lambda$ has $n-1$ distinct non-zero parts, so $\lambda_1'=n-1$. By Lemma \ref{split}, we have 
\begin{align*}&\sum_{(w,T)\in f^{-1}(\sigma)}t^{\frac{1}{2}(\ell(w)+\ell(wT)-|T|)}\,(1-t)^{|T|}\\&=\stacksum{(w,T)\in f^{-1}(\sigma)}{T_1=\emptyset}t^{N(w,T)}(1-t)^{|T|}\sum_{T'\::\:(w',T')\in{\mathcal A}(\Gamma(\lambda_1'))}t^{\frac{1}{2}(\ell(w')+\ell(w'T')-|T'|)}\,(1-t)^{|T'|}\,,\end{align*}
where $w':=wT$. By Proposition \ref{inv2}, the second sum in the 2-fold sum above is just $t^{\ell(C)}$, where $C$ is the first column of $\sigma$. Let $w''$ be the unique permutation in $S_n$ for which $w''[1,\lambda_1']=C$. The 2-fold sum above can be rewritten as follows:
\begin{align*}t^{\ell(C)}\sum_{T''\::\:(w'',T'')\in (f^r)^{-1}(\sigma)} t^{N(w'',T'')}(1-t)^{|T''|}&=t^{\ell(C)}t^{\cinv(\sigma)-\ell(C)}(1-t)^{\des(\sigma)}\\&=t^{\cinv(\sigma)}(1-t)^{\des(\sigma)}\,.\end{align*}
The first equality is the content of Proposition \ref{goleft}.
\end{proof}

\section{Appendix: The case of partitions with repeated parts (with A. Lubovsky)}\label{al}

We now consider the case of arbitrary partitions $\lambda$ with at most $n-1$ parts (possibly with repeated parts). In this general case, we will show that Ram's version of Schwer's formula leads to a new formula for $P_\lambda(X;t)$, which is analogous to the HHL formula (\ref{hlqform}), and specializes to it when $\lambda$ has $n-1$ distinct non-zero parts. To this end, we need to consider a subset of $\F(\lambda,n)$.

\begin{definition}
Let $\Fh(\lambda,n)$ be the subset of $\F(\lambda,n)$ containing all 
fillings in $\F(\lambda,n)$ for which their rightmost column is a concatenation of increasing segments corresponding to the positions $a,a+1,\ldots,b$ with $\lambda_a=\ldots=\lambda_b$. 
\end{definition}
Observe that if $\lambda$ has $n-1$ distinct non-zero parts, then $\Fh(\lambda,n)=\F(\lambda,n)$.



\begin{theorem}
  Given any partition $\lambda$ with at most $n-1$ parts, we have
  \label{mainresult}
  \begin{equation}\label{genform}
    P_{\lambda}(X;t)
    = \sum_{\sigma\in \Fh(\lambda,n)}
    t^{\cinv (\sigma )}\,(1-t)^{\des(\sigma)}\: x^{{\rm ct}(\sigma) } \,.
  \end{equation}
\end{theorem}

This theorem will be proved in Section \ref{ap1}, based on the compression results in Section \ref{pfmain} and a certain combinatorial bijection discussed in more detail in Section \ref{ap2}. 

\subsection{Related remarks}

In (\ref{pq}) we recalled that $P_\lambda(X;t)$ is obtained from $Q_\lambda(X;t)$ via division by a certain polynomial in $t$. Hence, it is natural to ask whether we can group the terms in the HHL formula (\ref{hlqform}) for $Q_\lambda(X;t)$, which are indexed by $\F(\lambda,n)$, such that the sum of each group is a term in (\ref{genform}) multiplied by the mentioned polynomial in $t$; here each group would have a representative from $\Fh(\lambda,n)$. The answer to this question is negative, as the following example shows. However, it would still be interesting to derive (\ref{hlqform}) from (\ref{genform}) or viceversa via a less obvious method. 

Let $n=4$ and $\lambda= (2,2,2,0)$. The collection of all fillings with content $(2,1,2,1)$ is 
\[
  \des=2 \quad
  \left\{\begin{array}{lll}
    \tableau{ {1}& {1}\\{3}& {2}\\ {4}& {3} }& 
    \tableau{ {3}& {2}\\{1}& {1}\\ {4}& {3} }& 
    \tableau{ {3}& {2}\\{4}& {3}\\ {1}& {1} }\\ 
    \cinv=0 &\cinv=1 & \cinv=2
  \end{array}
  \right.
\]
\[
  \des=1 \quad
  \left\{ \begin{array}{llllll}
    \tableau{ {1}& {1}\\{4}& {2}\\ {3}& {3} }& 
    \tableau{ {1}& {1}\\{3}& {3}\\ {4}& {2} }& 
    \tableau{ {3}& {3}\\{1}& {1}\\ {4}& {2} }& 
    \tableau{ {4}& {2}\\{1}& {1}\\ {3}& {3} }& 
    \tableau{ {4}& {2}\\{3}& {3}\\ {1}& {1} }& 
    \tableau{ {3}& {3}\\{4}& {2}\\ {1}& {1} }\\ 
    \cinv=1 &\cinv=1 &\cinv=2 &\cinv=2 &\cinv=3 &\cinv= 3
  \end{array}
  \right. \,.
 \]
By Theorem \ref{mainresult}, the coefficient of $x^{(2,1,2,1)}$ in $P_{\lambda}(X,t)$  is a sum of two terms:
\[t^0(1-t)^2 + t^1(1-t)^1 = t^0(1-t)^1\,.\]
By (\ref{hlqform}), the coefficient of $x^{(2,1,2,1)}$ in $Q_{\lambda}(X,t)/(1-t)^{\ell(\lambda)}$ is a sum of nine terms:
\[(1-t)^2(t^0+t^1+t^2) + (1-t)^1(2t^1 + 2 t^2 + 2t^3)=t^0(1-t)^1 [3]_t! \,. \]

\subsection{The proof of Theorem {\rm \ref{mainresult}}}\label{ap1}

The new ingredient for this proof is the following bijection, to be discussed in detail in Section \ref{ap2}. The notation is the same as in the previous sections. Given $w$ in $S_n$, we denote by $\overline{w}$ the lowest coset representative in $wS_n^\lambda$.

\begin{proposition}
  \label{bij}
There is a bijection $\phi:\A^L \to \A^R$ where
  \[ 
\A^L:=  \left\{ (w,T)\in \mathcal{A}^r(\rev(\Gamma))\::\: T_1 = \emptyset,\, wT \in S_n^{\lambda} \right\}\,,\;\;
\A^R:=  \left\{ (u,V) \in \mathcal{A}^r(\rev(\Gamma))\::\: V_1=\emptyset,\, u\in S_n^{\lambda} \right\}\,.
  \]
 Furthermore, if  $(w,T) \stackrel{\phi}{\mapsto} (u,V)$, then we have $u=\overline{w}$, $|V|=|T|$, and
  \[
  N(w,T)=N(\overline{w},V) - (\ell(w) - \ell(\overline{w}))\,, \qquad  \mathrm{ct}(w,T) = \mathrm{ct}(\overline{w},V)\,.
  \]
\end{proposition}

\begin{proof}[Proof of Theorem {\rm \ref{mainresult}}]
We have 
\begin{align*}
  P_\lambda(X;t)&=  \sum_{(w',T')\in \mathcal{A}(\Gamma) } 
  t^{\frac{1}{2} (\ell(w') + \ell(w'T') -|T'|) }\, (1-t)^{|T'|}\, 
  x^{\ct(w',T')} \\
  &=   \sum_{\substack{(w',T') \in \mathcal{A}(\Gamma) \\ T'_1=\emptyset}} t^{N(w',T')}\,(1-t)^{|T'|}\,
  x^{\ct(w',T')}
\sum_{T'':(w,T'') \in \mathcal{A}(\Gamma(\lambda'_1)) } t^{\frac{1}{2}(\ell(w) + \ell(wT'') -|T''|)} \,(1-t)^{|T''|}\\
&=\sum_{\substack{(w,T) \in \mathcal{A}^r(\Gamma^r)\\ T_1=\emptyset, \,wT\in S_n^\lambda}} t^{N(w,T)}\,(1-t)^{|T|}\,
  t^{\ell(w[1,\lambda'_1]) +\ell(w[\lambda'_1+1,n] )} \,
  x^{\ct(w,T)} \,.
\end{align*}
Here the first equality is Schwer's formula (\ref{hlpform}), the second one is based on Lemma \ref{split}, where $w:=w'T'$, and the third one is based on Proposition \ref{inv2}, where $T=\mathrm{rev}(T')$.

Let us now apply the bijection $(w,T) \stackrel{\phi}{\mapsto} (\overline{w},V)$ in Proposition \ref{bij}. By a property of this bijection, we have
\begin{align*}N(w,T)+\ell(w[1,\lambda'_1]) +\ell(w[\lambda'_1+1,n])&=N(\overline{w},V) - (\ell(w) - \ell(\overline{w}))+\ell(w[1,\lambda'_1]) +\ell(w[\lambda'_1+1,n])\\&=N(\overline{w},V)+\ell(\overline{w}[1,\lambda'_1])\,.\end{align*}
The last equality follows from the fact that
\[\ell(w) - \ell(w[1,\lambda'_1]) -\ell(w[\lambda'_1+1,n]) 
  =  \ell(\overline{w}) - \ell(\overline{w}[1,\lambda'_1]) -\ell(\overline{w}[\lambda'_1+1,n])\,, \]
where $\ell(\overline{w}[\lambda'_1+1,n])=0$. Hence, we can rewrite once again $P_\lambda(X;t)$ as
\begin{eqnarray*}
  P_\lambda(X;t)=
  \sum_{\substack{(u,V) \in \mathcal{A}^r(\Gamma^r)\\ V_1=\emptyset,\, u \in S_n^\lambda}} 
  t^{N(u,V) + \ell(u[1,\lambda'_1])}\,(1-t)^{|V|}\,
  x^{\ct(u,V)} \,;
\end{eqnarray*}
here we also used the other properties of the bijection in Proposition \ref{bij}.

 We now break up the sum in the right-hand side into smaller sums over the pairs $(u,V)$ whose image under $f^r$ is the same filling $\sigma$; here $u$ is the unique permutation in $S_n^\lambda$ for which $u[1,\lambda_1']$ coincides with the rightmost column of $\sigma$. Based on Proposition \ref{goleft}, we obtain
  \begin{align*}
P_\lambda(X;t)&=    \sum_{\sigma \in \Fh(\lambda,n)} 
\left(
t^{\ell(u[1,\lambda'_1])} 
\sum_{\substack{V:(u,V) \in (f^r)^{-1}(\sigma)} } 
    t^{N(u,V)} \,(1-t)^{|V|} 
    \right) x^{\ct(\sigma)} \\
    &=   
    \sum_{\sigma \in \Fh(\lambda,n)}  
    t^{\ell(u[1,\lambda'_1])}\, t^{\mathrm{cinv}(\sigma)-\ell(u[1,\lambda'_1])}\,
    (1-t)^{\mathrm{des}(\sigma)}\,x^{\ct(\sigma)}\\
&= \sum_{\sigma \in \Fh(\lambda,n)}  t^{\mathrm{cinv}(\sigma)}\,
    (1-t)^{\mathrm{des}(\sigma)}\,
     x^{\ct(\sigma)} \,.
  \end{align*}
\end{proof}

\subsection{The bijection in Proposition {\rm \ref{bij}}}\label{ap2}

We now get back to Proposition \ref{bij}, whose proof is based on Lemma \ref{lem} below. We start by preparing the background needed for the mentioned lemma.

Given a cycle $a=(a_1, \ldots, a_n)$ and a permutation $b$, we let $a^b:= bab^{-1}=(b(a_1),\ldots, b(a_n))$. The following simple identities will be needed below:
  \begin{align}
    &(j,j+1)(m,j+1)(m,j) = (m,j+1) \label{rel}\\
    &(j,j+1)(m,j)(m,j+1)=(m,j) \nonumber\\
    &(m,j)(j,j+1)=(m,j+1)(m,j)  \nonumber \\
    &(m,j+1)(j,j+1) = (m,j)(m,j+1)  \nonumber\,.
  \end{align}

In the sequel we will use both sequences of transpositions and the compositions of these transpositions (as permutations); the distinction will be made based on the context.
Recall the sequence of transpositions $\Gamma^r(k)$, which was interpreted in Section \ref{specschwer} as a reversed $\omega_k$-chain in the root system $A_{n-1}$:
\begin{eqnarray*}
 ( \, (k,k+1),& (k,k+2), \dots&, (k,n), \nonumber \\
   (k-1,k+1),& (k-1,k+2),  \dots &, (k-1,n),\nonumber\\
               & \vdots & \nonumber \\
   (1,k+1),& (1,k+2),  \dots & , (1,n) \, )\,.\nonumber
\end{eqnarray*}
It turns out that the following sequence $\widetilde{\Gamma}^r(k)$ is also a reversed $\omega_k$-chain:
\begin{eqnarray*}
 ( \, (k,k+1),& (k-1,k+1), \dots&, (1,k+1), \nonumber \\
   (k,k+2),& (k-1,k+2),  \dots &, (1,k+2),\nonumber\\
               & \vdots & \nonumber \\
   (k,n),& (k-1,n),  \dots & , (1,n) \, )\,.\nonumber
\end{eqnarray*}
Indeed, we can get from $\Gamma^r(k)$ to $\widetilde{\Gamma}^r(k)$ by swapping pairs corresponding to commuting transpositions (see \cite{lapawg}).
We will use the convention of writing $(b,a)$ for transpositions $(a,b)$ in $\widetilde{\Gamma}^r(k)$. 
Recall that we have the reversed $\lambda$-chain $\mathrm{rev}(\Gamma)= \Gamma^r_1\dots\Gamma^r_{\lambda_1}$, where $\Gamma_i^r:=\Gamma^r(\lambda'_i)$. 
For $1 \leq j \leq \lambda'_1$
  we will consider the new reversed $\lambda$-chain 
  $\mathrm{rev}(\Gamma)_j:= \widetilde{\Gamma}^r_1\dots\widetilde{\Gamma}^r_{\lambda_j}\Gamma_{\lambda_j+1}^r\dots \Gamma_{\lambda_1}^r$ where
  $\widetilde{\Gamma}^r_i:=\widetilde{\Gamma}^r(\lambda'_i)$.
  The relationship between $\Gamma^r(k)$ and $\widetilde{\Gamma}^r(k)$ gives a relationship between $\rev(\Gamma)$
  and $\rev(\Gamma)_j$, which leads to an obvious bijection between the corresponding admissible sets. Therefore in the following
  proof we use $\rev(\Gamma)_j$, but based on the mentioned bijection the results are easily translated to $\rev(\Gamma)$.
  The splitting of $\mathrm{rev}(\Gamma)_j$ induces the splitting $T=T_1\dots T_{\lambda_1}$ of any $T$ for which 
  $(w,T) \in \mathcal{A}^r(\mathrm{rev}(\Gamma)_j)$.

\begin{lemma}
  \label{lem}
  If $\lambda_j=\lambda_{j+1}$ for a fixed $j$, then there is a 
  bijection $\phi_j:\A_j^L \to \A_j^R$,
where
  \[\A_j^L:=\{(w,T)\in \A^r(\rev(\Gamma)_j) \::\:
  w(j)> w(j+1),\, wT(j) < wT(j+1),\, T_1=\emptyset\}
  \]
and
\[\A_j^R := \{(w',T')\in \A^r(\rev(\Gamma)_j) \::\:
  w'(j)< w'(j+1),\, w'T'(j) > w'T'(j+1) ,\, T'_1=\emptyset\}\,.\]
  If $(w,T) \mapsto (w',T')$ under this bijection, then  $w'=w(j,j+1)$, $|T|=|T'|$, $\ct(w,T)=\ct(w',T')$, and
  \[
 N(w,T) = N(w',T')   -1 \,.
  \]
\end{lemma}

\begin{proof}
  Define the map from $\A_j^R$ to $\A_j^L$ by 
  \begin{eqnarray*}
    (w',T') \mapsto (w:=w'(j,j+1),T) 
  \end{eqnarray*}
  where $T$ is described in the sequel. 

  Consider $(w',T')\in \A_j^R$, and recall the splitting $T'=T_1\dots T_{\lambda_1}$, which is induced by the splitting of 
  $\mathrm{rev}(\Gamma)_j:= \widetilde{\Gamma}^r_1\dots\widetilde{\Gamma}^r_{\lambda_j}\Gamma_{\lambda_j+1}^r\dots \Gamma_{\lambda_1}^r$. If 
  $T'= \left( (a_1,b_1), \dots , (a_p,b_p) \right)$, we define  
  $(T')^{(j,j+1)}:= \left( (a_1,b_1)^{(j,j+1)}, \dots, (a_p,b_p)^{(j,j+1)} \right)$.
  If for all $i$, the sequence $T_i$ does not contain the 
  transpositions $(m,j)$ and $(m,j+1)$ adjacent to one another for some $m$, then define
  $T:= (T')^{(j,j+1)}$. We have $wT=w'T'(j,j+1)$, hence $(w,T)\in \A_j^L$, and it is clear that $|T| = |T'|$.
  The filling $f^r(w,T)$ is obtained from the filling $f^r(w',T')$ by switching row $j$ with row $j+1$, hence
  $\ct(w,T)=\ct(w',T')$.

  To facilitate with the remainder of the proof we  introduce some terminology.
  In order to construct $T$, we define a new splitting of $T'$ (different from $T'=T_1\dots T_{\lambda_1}$). 
 We scan $T'$ from left to right, one  $T_i$ at a time, paying attention to certain transpositions.
  We define a \emph{start marker} as a sequence of adjacent transpositions 
 $((m,j+1),(m,j))$ in $T_i$ for $i\le \lambda_j$, and denote it by $s_m$. Similarly, we define an 
 \emph{end marker} simply as a transposition $(m,j)$ in $T_i$ for $i\le \lambda_j$, 
 and denote it by $e_m$.
 If $i> \lambda_j$, a  {start marker}  is a sequence
 $((m,j),(m,j+1))$ in $T_i$, also denoted by $s_m$, and an 
 {end marker}  is a transposition $(m,j+1)$ in $T_i$, 
 also denoted by $e_m$.

 If $T'$ contains a {start marker} $s_m$ 
 and we split $T'$ as $T'=Us_mV$, then
 $w'U(j)<w'U(j+1)$, because otherwise $T'$ is not $w'$-admissible. By our hypothesis $w'T'(j)>w'T'(j+1)$, hence there is a
 transposition $\tau$ in $V$ determining the splitting
 $T=Us_mW\tau R$,
 where $w'Us_mW\tau(j)>w'Us_mW\tau(j+1)$. Consider the first such $\tau$ in $V$, which means that $\tau$ is an end marker $e_k$ and $w'Us_mW(j)<w'Us_mW(j+1)$. 
 We call this $e_k$ the \emph{complementary end marker} to $s_m$. 
 Observe that a {complementary end marker} can be in $T_i$, $i>\lambda_j$ for a start marker in $T_{i'}$, $i' \leq \lambda_j$.

  Consider the following splitting of $T'$: 
  \begin{equation}
    T'= U_1 s_{m_1} W_1 e_{k_1} U_2 \dots U_{t-1}  s_{m_t} W_t e_{k_t} U_{t+1}\,,
    \label{phiTsplit}
  \end{equation}
where, in our scanning of $T'$, $s_{m_1}$ is the first {start maker}, $e_{k_1}$ is its {complementary end marker}, $s_{m_2}$ is the next start marker, etc., and $U_{t+1}$ contains no start markers. 
Observe that
$(j,j+1)s_{m_i}We_{k_i}(j,j+1)= e^{(j,j+1)}_{m_i}Ws_{k_i}$, by (\ref{rel}). Based on this, we define the sequence of transpositions
\begin{equation}
  T:= U_1^{(j,j+1)} e^{(j,j+1)}_{m_1} W_1 s_{k_1} U_2^{(j,j+1)} \dots U_{t-1}^{(j,j+1)}  
  e^{(j,j+1)}_{m_t} W_t s_{k_t} U_{t+1}^{(j,j+1)}\,,
  \label{phiTsplit2}
\end{equation}
in order to have $(j,j+1)T'(j,j+1)=T$ as a composition of transpositions. It is not hard to check that $T$ is $w$-admissible, where $w=w'(j,j+1)$. 
Therefore $wT=w'T'(j,j+1)$ and thus $(w,T)\in \A_j^L$. It is also clear that $|T|=|T'|$.
Furthermore, we have
$\ct(w,T)=\ct(w',T')$, since, for each column, the only difference that can occur between the 
fillings $f^r(w,T)$ and $f^r(w',T')$ is due to interchanging the values in rows  $j$ and $j+1$.  
The fact that $N(w,T) = N(w',T')-1$ follows from (\ref{tailb}).

To show that the constructed map is invertible, we apply a similar procedure to $T$, by scanning it from right to left; here the notion of an end marker needs to be slightly modified, in the obvious way. It is not hard to see that this procedure determines a splitting of $T$ at the same positions as those in the previous splitting of $T'$, so we recover the original sequence $T'$.
\end{proof}


\begin{proof}[Proof of Proposition {\rm \ref{bij}}] 
  \label{bijproof}
Given $(w,T)\in \A^L$, we can biject it to $(\overline{w},T')\in\A^R$ by applying some procedure to get from $w$ to $\overline{w}$ in $S_n^\lambda$ by swapping adjacent positions $j$ and $j+1$; then for each such swap, we apply the bijection $\phi_j$ in Lemma \ref{lem}. The reorderings have to occur in the intervals $[a,b]$ for which $\lambda_a=\ldots=\lambda_b$. For instance, the mentioned intervals can be considered from top to bottom, and at each step, for a certain interval $[a,b]$, we can define $j$ to be the first descent in $[a,b)$ of the corresponding permutation $u$; in other words, we set
\begin{equation}\label{rule1}j:=\min\,\{i\in[a,b)\::\:u(i)>u(i+1)\}\,.\end{equation}

In order to apply the same sequence of transformations in reverse order to $(\overline{w},T')$, we need to use the similar procedure described below. Start with the permutation $\overline{w}T'$ and convert it to a permutation in $S_n^\lambda$ as follows, by considering the intervals $[a,b]$ from bottom to top; then for each swap of adjacent positions $j$ and $j+1$, apply the bijection $\phi_j^{-1}$. If the current interval is $[a,b]$ and the current permutation is $u$, then $j$ is defined as follows:
\begin{equation}\label{rule2}j:=u^{-1}(k)\,,\;\;\;\;\mbox{where }\:k:=\max\,\{u(i)\::\:i\in[a,b),\:u(i)>u(i+1)\}\,.\end{equation}

The fact that the second procedure reverses the first one rests on the following simple observation. Fix $\pi$ in $S_n$ and sort its entries by swapping adjacent entries, i.e., by applying adjacent transpositions $s_j=(j,j+1)$ on the right. Assume that $j$ is chosen as in (\ref{rule1}), where $[a,b)=[1,n)$. The mentioned sorting of $\pi$ is realized by the unique reduced word for $\pi^{-1}$ of the form 
\[\ldots (s_is_{i-1}\ldots)(s_{i+1}s_{i}\ldots)\ldots(s_{n-1}s_{n-2}\ldots)\,.\]
Now consider the sorting of $\pi^{-1}$ given by the reversed reduced word (for $\pi$)
\[(\ldots s_{n-2}s_{n-1})\ldots(\ldots s_is_{i+1})(\ldots s_{i-1}s_i)\ldots\;.\]
It is not hard to see that this sorting is realized by successively applying $s_j$ with $j$ given by (\ref{rule2}), where $[a,b)=[1,n)$.
\end{proof}

\begin{remark}
It is clear that the bijections $\phi_j$ in Lemma \ref{lem} satisfy $\phi_i\phi_j = \phi_j\phi_i$ for $j\ge i+2$. If one can show that they also satisfy the braid relations
$\phi_j\phi_{j+1}\phi_j = \phi_{j+1}\phi_j\phi_{j+1}$, whenever these compositions are defined, then the bijection in the proof of Proposition \ref{bij} does not depend
on the order in which we apply the maps $\phi_j$, to go from $(w,T)$ to $(\overline{w},T')$.
\end{remark}


\end{document}